\documentclass[12pt]{article}
\usepackage{graphicx}
\usepackage[dvips]{epsfig}
\usepackage{subfigure}
\usepackage[a4paper]{geometry}
\usepackage{tikz-cd}
\usepackage{amsmath,amssymb,amsthm,amsfonts,hyperref,amscd}
\usepackage{times}
\usepackage{setspace}
\usepackage{verbatim}
\usepackage{cancel}
\usepackage{mathtools}
\usepackage[toc,page]{appendix}
\usepackage{multirow}

\newtheorem{thm}{Theorem}[section]
\newtheorem{defn}[thm]{Definition}
\newtheorem{lem}[thm]{Lemma}

\newtheorem{cor}[thm]{Corollary}
\newtheorem{rmk}[thm]{Remark}

\theoremstyle{plain}
\theoremstyle{definition}
\counterwithin*{equation}{section}

\newcommand{\F}{\mathcal{F}}

\newcommand{\nD}{\nabla}

\newcommand{\KN}{\mathbin{\bigcirc\mspace{-15mu}\wedge\mspace{3mu}}}

\newcommand{\addresses}{\bigskip\footnotesize

J. Moon, \textsc{Department of Mathematics, Chung-Ang University, 84 HeukSeok-ro DongJak-gu, Seoul 06974, Republic of Korea.} \par\nopagebreak
\textit{E-mail address:}\texttt{dsfish999@cau.ac.kr} }

\title{On a compact Einstein-type manifold with Riemannian foliations}
\author{Jungwoo Moon}
\date{}

\begin{document}

\maketitle

\begin{abstract}
We investigate a compact Einstein-type manifold whose potential vector field generates a Riemannian foliation. In particular, we prove necessary conditions for such a manifold to be taut and to have a splitting property. Additionally, some properties of taut Riemannian foliations on a compact almost Ricci solitons and a compact Einstein manifolds are provided.
\end{abstract}

\section{Introduction} An Einstein manifold is a Riemannian manifold $M$ whose Ricci curvature is a constant multiple of its metric $g.$ Due to its role in differential geometry, the study of Einstein manifolds has been extensively considered in Riemannian geometry. To generalize the notion of Einstein manifolds, Catino et al. \cite{CMMR} introduced the Catino-Mastrolia-Monticelli-Rigoli-Einstein type manifold(hereafter, \textit{Einstein-type manifold} in short) defined by the following equation on $M$.
\begin{equation}
    \alpha Ric+\dfrac{\beta}{2}L_Xg+\gamma X\otimes X=\lambda \text{ } g,
\end{equation} for a potential vector field $X$, constants $\alpha,\beta,\gamma$ and a smooth function $\lambda$ on $M$. Here, $g$ denotes the Riemannian metric of $M$, $Ric$ the Ricci curvature of $(M,g)$, $L_Xg$ the Lie derivative of $g$ with respect to $X,$ and $\otimes$ the tensor product of tensor fields on $M$. When $X=\nD f$ for some smooth function $f,$ we say a Einstein-type manifold is \textit{gradient}.

In particular, if $\alpha,\beta\neq 0,$ $\gamma=0$, then $(M,g)$ is called an \textit{almost Ricci soliton}. If additionally, $\lambda$ is constant, then $(M,g)$ reduces to a \textit{Ricci soliton}. On the other hand, if $\alpha=\gamma=0$ and $\lambda$ equals the scalar curvature $S$ of $M,$ then $(M,g)$ is called a \textit{Yamabe soliton.}

Now we recall several properties for gradient Einstein-type manifolds. Fernandez-Lopez and Garcia-Rio proved a rigidity theorem on complete gradient Ricci soliton with vanishing Cotton tensor $C$ \cite{FG}. Motivated by their work, Catino, Mastrolia and Monticelli established a rigidity result for complete gradient Ricci soliton with vanishing complete divergence $\mathrm{div}^3C$ of $C$ \cite{CMM}. Subsequently, Catino and Mastrolia have shown that $\mathrm{div}^2C(X)=0$ is equivalent to $\mathrm{div}^3C=0$ on a compact Ricci soliton \cite{CM}. Co and Hwang further proved that an analogous statement holds for a gradient almost Ricci soliton \cite{CH}. 

More generally, Catino et al. have proven that a complete, noncompact, non-conformally Einstein, gradient Einstein-type manifold of dimension $\geq 3$ with vanishing Cotton tensor and a proper potential function $f$ is locally a warped product with a $1$-codimensional Einstein fiber over each regular level set of $f$ (cf. Theorem 2.2 for the details) \cite{CMMR}. In particular, it should be remarked that if every level set of $f$ is regular, then the $1$-codimensional Einstein fiber could also be regarded as the transverse space of the foliation induced by $\nD f$ whose normal bundle is involutive. That is, the Lie bracket is a closed operation on the normal bundle of $\nD f$ (cf. Remark 2.3).

Although examples of almost Ricci solitons whose potential vector field $X$ is neither non-zero nor gradient are known (cf.\cite{BD, Lau}), few properties on a general Einstein-type manifold have been established. Hence, the purpose of this paper is to investigate properties of compact Einstein-type manifolds that are analogous to known results for gradient cases. 

To study a Einstein-type manifold $(M,g)$ with dimension $\geq 3$, we assume that the potential vector field $X$ defines a Riemannian flow (i.e. $1$-dimensional Riemannian foliation) $\mathcal{F}_X$ and denote the resulting manifold by $(M,\mathcal{F}_X,g)$. That is, $X$ is nonvanishing and satisfies $L_Xg_Q=0,$ where $g_Q$ is the Riemannian metric restricted to the normal bundle of $X$. For simplicity, we further assume throughout this paper that $M$ is compact and the mean curvature $\kappa$ of $(M,\mathcal{F}_X,g)$ is basic, i.e. $L_X\kappa=0$ (cf. Lemma 2.5). Since $\F_X$ is $1$-dimensional, $\F_X$ may be viewed as an analog of the foliation generated by the potential function of a gradient Einstein-type manifold.

Under the above setting, we obtain the following property on Einstein-type manifold with the foliation $\F_X$ (cf. Theorem 4.1).

\begin{thm}
    Let $(M,\mathcal{F}_X,g)$ be a compact Einstein-type manifold of dimension $\geq 3$ satisfying $\beta\neq 0$. If the mean curvature $\kappa$ of $\mathcal{F}_X$ is basic, then $\F_X$ is taut. i.e. $\F_X$ is minimal for some Riemannian metric $g$.
\end{thm}

In particular, we additionally assume that the normal bundle of $\F_X$ is involutive to study the analogous converse of the result of Catino et al. \cite{CMMR} on compact Einstein-type manifolds (cf. Remark 2.12). Then we have the following corollary of Theorem 1.1, which explains the splitting behavior of the total space $M$ (cf. Corollary 4.2).

\begin{cor}
     Let $(M,\mathcal{F}_X,g)$ be a compact, $n(\geq3)$-dimensional Einstein-type manifold with $\alpha\beta\neq0$. If the normal bundle of the Riemannian flow $\mathcal{F}_X$ is involutive and $\kappa$ of $\F_X$ is basic, then the Riemannian universal covering $\tilde{M}$ of $M$ splits as the Riemannian product $\mathbb{R}\times N$, where $N$ is a complete Einstein manifold.  
\end{cor} 

Note that the splitting property of Corollary 1.2 holds if $\alpha=0$ and $\beta \neq 0,$ but the manifold $N$ of the transverse part of $M$ may not be Einstein (cf. Remark 4.3).

Another application of Theorem 1.1 provides a taut Riemannian flow on an almost Ricci soliton. i.e. a Einstein-type manifold satisfying $\alpha=\beta=1$ and $\gamma=0$. Precisely, we have the following corollary (cf. Remark 5.1).

\begin{cor}
Let $(M,\F_X,g)$ be a compact almost Ricci soliton of dimension $\geq 3$. If $\kappa$ of $\F_X$ is basic, then $M$ is Einstein and $\F_X$ is a taut Riemannian flow.
\end{cor} 

Note that the above results are based on $\beta\neq 0$ case. In particular, no aforementioned results are justified on a compact Einstein manifold, which is also an Einstein-type manifold with $\beta=0$ and $\gamma=0$. Hence, it is natural to ask when a Riemannian flow on a compact Einstein manifold is taut since defining a potential vector field is redundant on an Einstein manifold. For the aforementioned reason, let us consider a compact Einstein manifold with a transverse Einstein Riemannian flow. Here, a \textit{transverse Einstein foliation} is a Riemannian foliation whose transverse metric is Einstein, i.e. $Ric^Q=\lambda^Q g_Q$ for some constant $\lambda^Q$. 

Previously, A. Ranjan \cite{Ran} has proven that a compact Einstein manifold with negative scalar curvature cannot attain a Riemannian flow (cf. Lemma 5.2). Also, applying the author's result \cite{Moo}, one could obtain that a $1$-dimensional transverse Einstein foliation with nonnegative transverse scalar curvature on a compact manifold must be taut (cf. Lemma 5.3). Unifying the result of the author and Ranjan, we conclude that such Riemannian flow on a compact Einstein manifold must be taut (cf. Theorem 5.4). Furthermore, the above unified result of the author and Ranjan is generalized as follows (cf. Theorem 5.5).
   
\begin{thm} A transverse Einstein foliation with zero scalar curvature leaves on a compact Einstein manifold must be taut. \end{thm} 

The content of this paper consists of the following. We recall some basic calculations and backgrounds in Section 2. In section 3, some necessary preliminaries of Riccati equation with respect to the Riemnnanian foliation are recalled. In section 4, we prove Theorem 1.1 and Corollary 1.2. In the last section, Corollary 1.3 and Theorem 1.4 are proved.

The author would like to thank to Professor Seungsu Hwang and Professor Seoung Dal Jung for their helpful discussions and plenty of kind advises to complete this work.

\section{Preliminaries}

Throughout this paper, $M$ always denotes an $n(\geq 3)$-dimensional complete Riemannian manifold with metric $g$ and corresponding {Levi-Civita connection} $D$. Then for vector fields $Y,Z,V$ on $M,$ we have the following identity.
\begin{equation}
    g(D_{Z}Y,V)=\dfrac{1}{2}\left(g(D_{Z}Y,V)+g(D_{V}X,Z)\right)-\dfrac{1}{2}\left(g(D_{Z}Y,V)-g(D_{V}Y,Z)\right).
\end{equation}

We would like to focus on the former term of the right hand side in (2.1), which is called the {symmetric part} of $g(D_{Z}Y,V)$. Hereafter, we denote the symmetric part of $g(D_{Z}Y,V)$ by $\displaystyle{\dfrac{1}{2}L_Yg(Z,V)}$ by definition of the Lie derivative with respect to $Y$ (see e.g. \cite{Bes,Ton}).

Now, the {Riemann curvature} on $M$ is defined as follows.
\begin{equation}
    R(e_i,e_j,e_k,e_l)=g(D_{e_j}D_{e_i}e_k-D_{e_i}D_{e_j}e_k+D_{[e_i,e_j]}e_k,e_l),
\end{equation} where $\{e_i\}$ is a local moving frame on $M$ with respect to the metric $g$ and $[\cdot,\cdot]$ is the Lie bracket on $M$.
 Hereafter, we denote a symmetric $2$-tensor $Ric$ by the {Ricci curvature} of $M$, $S$ by the {scalar curvature} of $M$, and $z$ by the {traceless Ricci curvature} of $M$. 
 With these settings, we have the definition of {Cotton tensor} and {Weyl curvature} as follows, respectively. 

 \begin{equation}
  W=R-\left(\dfrac{1}{n-2}z\KN g+\dfrac{S}{2n(n-1)}g\KN g\right),  
 \end{equation} 
 \begin{equation}
 C=\mathrm{d}^D\left(Ric-\dfrac{S}{2(n-1)}g\right),
 \end{equation}where $\KN$ is the Kulkarni-Nomizu product defined by
 \begin{equation}
 \begin{split}
     T_1\KN T_2(e_i,e_j,e_k,e_l)&=T_1(e_i,e_k)T_2(e_j,e_l)+T_1(e_j,e_l)T_2(e_i,e_k)\\&-T_1(e_i,e_l)T_2(e_j,e_k)-T_1(e_j,e_k)T_2(e_i,e_l)
 \end{split}
 \end{equation} and $\mathrm{d}^D$ is the Codazzi operator given by
 \begin{equation}
     \mathrm{d}^DT_1(e_i,e_j,e_k)=D_{e_i}T_1(e_j,e_k)-D_{e_j}T_1(e_i,e_k)
 \end{equation}for symmetric $2$-tensors $T_1$ and $T_2$ \cite{Bes}.
 Also, to justify the later claim, we define the tensor inner product. For two $k$-tensor fields $T_1$ and $T_2$ on $M$, $g(T_1,T_2)$ is defined as follows.
\begin{equation}
    g(T_1,T_2)=\sum T_1(e_{i_1},...e_{i_k}) T_2(e_{i_1},...e_{i_k}).
\end{equation}
     
It is well known that the divergence of Weyl curvature is a constant multiple of the Cotton tensor. To be precise, we have the following formula.
\begin{equation}
C(e_k,e_l,e_j)=\dfrac{n-3}{n-2}\sum_{l=1}^nD_{e_i}W(e_i,e_j,e_k,e_l).
\end{equation}

As Weyl curvature is trace-free, it should be remarked that Cotton tensor is {trace-free} with respect to any pair of its two entries. Also, the following remark is well-known.
\begin{rmk}\normalfont
    For a Riemannian manifold $M$ of dimension $\geq 4,$ we say the Weyl curvature is harmonic if is satisfies 
    \begin{equation}
C(e_k,e_l,e_j)=\dfrac{n-3}{n-2}\sum_{l=1}^nD_{e_i}W(e_i,e_j,e_k,e_l)=0.
\end{equation} i.e. harmonic Weyl tensor is equivalent to vanishing Cotton tensor. 
\end{rmk}

Now let us recall the definition of a Einstein-type manifold. A Einstein-type manifold is a Riemannian manifold with satisfying the following condition on metric tensor.
\begin{equation}
    \alpha\text{ } Ric+\dfrac{\beta}{2}\text{ }L_Xg+\gamma \text{ }X\otimes X=\lambda\text{ } g, 
\end{equation} where $\alpha, \beta, \gamma$ are real constants, $X$ is the potential vector field of the Einstein-type manifold, and $\lambda$ is a smooth function on $M$. For later contents, we would assume that $X$ may not be conservative. That is, the given Einstein-type manifold may not be \textit{gradient}. This is because the following result of Catino et al. on a gradient Einstein-type manifold is already known \cite{CMMR}.
\begin{thm}
    Let $(M,g)$ be a complete, noncompact gradient Einstein-type manifold of dimension $\geq 3$ and let $f$ be its potential function. Assume that the inverse image $f^{-1}(a)$ for each real number $a$ is compact and neither $\beta = 0$ nor $\beta^2 = (n-2)\alpha\gamma.$
    If the Cotton tensor of $(M,g)$ vanishes, then $M$ is locally a warped product of a curve with $1$-codimensional Einstein fiber on every regular level set of $f$.
\end{thm}

\begin{rmk}\normalfont
 It should be remarked that $1$-codimensional Einstein fibers in Theorem 2.2 are integrable since $\ker df$ is closed under the Lie bracket (i.e. it is \textit{involutive}). As a consequence, the fibers are the integrable hypersurfaces on $M$ by definition of integral submanifolds. Hence, if we consider the flow $\F$ along $\nD f,$ then $\F$ has an involutive normal bundle. 
\end{rmk}

To observe Einstein-type manifolds with nonvanishing potential vector field $X$, we need to review the preliminaries of transverse geometry. Recall that a foliation is a cover of $M$ whose elements are integral submanifolds. In particular, if such submanifold is $1$-dimensional, then the given foliation forms a global flow on $M$. Hence, we need to assume that the tangent vector field of a global flow is nonvanishing. 

Throughout this paper, we assume that $X$ is a nonvanishing potential vector field and the integral curves of $X$ forms a $1$-dimensional foliation $\mathcal{F}_X$ on a compact Einstein-type manifold $M$ and let us call $\mathcal{F}_X$ the \textit{potential flow}. Also, to establish the geometry on a normal bundle of the potential flow $\mathcal{F}_X$, we recall the  definition of Riemannian foliation on a Riemannian manifold $M$ as follows.
\begin{defn}
 Let $\F$ be a foliation on a Riemannian manifold $M$. $\F$ is said to be Riemannian if all vector fields $U$ tangent to $\F$ satisfies $L_Ug(Y,Z)=0$ for any vector fields $Y,Z$ perpendicular to $\F$.  
\end{defn}
Due to the result in \cite{Rei}, any Riemannian manifold with Riemannian foliation has a \textit{bundle-like metric} $g$. That is, $g$ is decomposed as $g=g_{\F}+g_Q,$ where $g_{\F}$ is the tangential component of $\F$ and $g_Q$ is the normal component of $\F$. Therefore, the given Einstein-type manifold is also a bundle-like metric as $\mathcal{F}_X$ is assumed to be the Riemannian. 

Similar to Riemannian geometry on a differential manifold, we define \textit{transverse Levi-Civita connection} $\nD$ and \textit{transverse Riemann curvature} $R^Q$ on a manifold with Riemannian foliation $(M,\F)$ as follows.
\begin{equation}
\begin{split}
&\nD_{Y}Z= \pi([Y,Z]) \text{ if } Y \text{ is tangent to } \F\\
&\nD_{Y}Z= \pi(D_YZ)\text{ if } Y \text{ is perpendicular to } \F,
\end{split}     
\end{equation} where $\pi$ is a projection from the tangent bundle $TM$ to the normal bundle $Q$ of a foliation $\F$ and $Z$ is a vector field perpendicular to $\F$ and

\begin{equation}
    R^Q(Y,Z,V,W)=g(\nD_{Z}\nD_{Y}V-\nD_{Y}\nD_{Z}V+\nD_{[Y,Z]}V,W),
\end{equation} where $Y,Z,V,W$ are vector fields perpendicular to $\F,$ from now on.

It should be remarked that a transverse Riemann curvature $R^Q$ is a \textit{basic} tensor. That is, the interior product $i_UR^Q$ and the Lie derivative $L_UR^Q$ with respect to a vector field tangent to $U$ must vanish \cite{Ton}.

Also, by taking trace on $R^Q$ with respect to $g_Q$, we have the Ricci curvature $Ric^Q$ with respect to $R^Q$, called the \textit{transverse Ricci curvature}, and the scalar curvature $S^Q=tr_{g_Q}Ric^Q$ of $Ric^Q$, called the \text{transverse scalar curvature}. Due to their definitions, $Ric^Q$ and $S^Q$ are basic curvatures on $(M,\F)$.

To remind the divergences and mean curvature on the geometry of $Q,$ the definition of O'Neill's $T$ and $A$ tensor should be recalled (\cite{Bes,Ton}). For vector fields $Y,Z$ on a manifold with a Riemannian foliation $(M,\F)$,
\begin{equation}
    T_YZ=\pi(D_{\pi^{\perp}Y}\pi^{\perp}Z)+\pi^{\perp}(D_{\pi^{\perp}Y}\pi Z),
\end{equation} and
\begin{equation}
   A_YZ=\pi(D_{\pi Y}\pi Z)+\pi^{\perp}(D_{\pi Y}\pi^{\perp}Z),
\end{equation} where $\pi^{\perp}$ is a projection from the tangent bundle $TM$ to the tangent bundle $T\F$ of a foliation $\F$.

Hence, $T$ is regarded as the generalization of the second fundamental form of a Riemannian foliation $\F$ on $M$. In particular, we define the \textit{mean curvature} of $\F$ as follows.
\begin{equation}
    \kappa^{\sharp}=\sum_{\alpha=1}^p T_{\epsilon_\alpha}\epsilon_\alpha,
\end{equation} where $\{\epsilon_\alpha\}$ is a local moving frame on $T\F$ with respect to $g$. Accordingly, we have the corresponding $1$-form of $\kappa^{\sharp}$ and it is denoted by $\kappa$.

In particular, a Riemannian foliation is called \textit{geodesible} if a manifold with a Riemannian foliation $(M,\F)$ has a bundle-like metric satisfying $T=0.$ Similarly, a Riemannian foliation is called \textit{taut} if a manifold with a foliation has a bundle-like metric satisfying $\kappa=0.$ Therefore, any geodesible Riemannian foliation is necessarily taut and any taut foliation is geodesible if the foliation is $1$-dimensional.

In general, tautness of a given Riemannian foliation on a compact Riemannian manifold is an important geometric property, as the transverse divergence does not coincide with the formal adjoint of $\nD$. Therefore, there are two divergences on the geometry of $(Q,g_Q)$ of a foliated manifold $(M,\F)$ with a Riemannian foliation, as follows.

\begin{equation}
    \text{div}_Q=\sum_{i=1}^qi_{e_i}\circ\nD_{e_i},
\end{equation} is called the \textit{transverse divergence}, where $\{e_i\}$ is a local moving frame on $Q$ with respect to $g$ and
 
\begin{equation}
    \text{div}_B=\text{div}_Q-i_{\kappa^{\sharp}},
\end{equation} is called the \textit{basic divergence,} which is identified with the divergence on $(M,g)$ for any basic tensor fields \cite{AL}. 

\begin{lem}\cite{AL,Ton}
 For any Riemannian foliation $\F$ on a compact Riemannian manifold $M$, we may assume the mean curvature $\kappa$ of $\F$ be basic and $\text{div}_B\kappa=0$. Hence, the cohomology class of $\kappa$ vanishes if and only if $\F$ is taut.
\end{lem}

It should be also remarked that a basic mean curvature $\kappa$ of a Riemannian foliation on compact manifolds should be a closed $1$-form \cite{Ton}. Therefore, the cohomology class $[\kappa]$ is well-defined.

There is an another way to determine tautness of a Riemannian foliation. For later use, we define the following basic symmetric $2$-tensor, called the \textit{symmetric tautness tensor}, as follows. 
\begin{equation}
    T_{\kappa}=\nD_{tr}\kappa-\kappa\otimes\kappa,
\end{equation} where $\nD_{tr}\kappa$ is the tensor derivative of $\kappa$ only on $Q.$ Then, the following property is proven.
\begin{lem}\cite{Moo}
Assume that the mean curvature $\kappa$ of a Riemannian foliation $\F$ is basic and $\text{div}_B\kappa=0$ on a compact manifold $M$. Then $\F$ is taut if and only if $T_{\kappa}=0.$
\end{lem}

It should be cautious that we may adjust the metric of leaves to obtain $\text{div}_B\kappa=0$ with $g_Q$ fixed. Thus, if we fix the metric of a compact Riemannian manifold with a Riemannian foliation $(M,\F,g)$, then the mean curvature $\kappa$ of $\F$ may not be basic. Furthermore, $\F$ may not satisfy $\text{div}_B\kappa=0$ although if we additionally assume that $\kappa$ is basic. As a consequence, the basic divergence of the symmetric tautness tensor $\text{div}_BT_{\kappa}$ is calculated by the following (cf. \cite{Moo}).

\begin{lem} Let $(M,\F,g)$ be a compact Riemannian manifold with a Riemannian foliation. If the mean curvature $\kappa$ of $\F$ is basic, then the following equation is computed. 
    \begin{equation}
    \text{div}_B(T_{\kappa})=i_{\kappa^{\sharp}}Ric^Q+d(\text{div}_B\kappa)-(\text{div}_B\kappa)\kappa.
\end{equation}
\end{lem}
\begin{proof}
    Recalling the work of Jung\cite{Jun}, we have the transverse Weitzenb\"ock formula as follows.
    \begin{equation}
        \Delta_B\kappa=-\text{div}_B(\nD_{tr}\kappa)+i_{\kappa^{\sharp}}Ric^Q+L_{\kappa^{\sharp}}\kappa-\nD_{\kappa^{\sharp}}\kappa,
    \end{equation} where $\Delta_B$ is the basic laplacian, defined by the composition of the exterior differential and the divergence as follows.
    \begin{equation}
         \Delta_B=-\text{div}_B\circ d-d\circ \text{div}_B.
    \end{equation} 
 Then, the following is implied, since $d\kappa=0$ and the Cartan formula holds.
 \begin{equation}
        -d(\text{div}_B\kappa)=-\text{div}_B(\nD_{tr}\kappa)+i_{\kappa^{\sharp}}Ric^Q+d|\kappa|^2-\dfrac{1}{2}d|\kappa|^2.
    \end{equation}
Thus, we obtain the computation of $\text{div}_BT_{\kappa}$ as follows. 
\begin{equation}
    \text{div}_BT_{\kappa}=i_{\kappa^{\sharp}}Ric^Q+\dfrac{1}{2}d|\kappa|^2+d(\text{div}_B\kappa)-(\text{div}_B\kappa)\kappa-\nD_{\kappa^{\sharp}}\kappa.
\end{equation}
That is, \begin{equation}
    \text{div}_BT_{\kappa}=i_{\kappa^{\sharp}}Ric^Q+d(\text{div}_B\kappa)-(\text{div}_B\kappa)\kappa,
\end{equation} as desired.
\end{proof}

Specifically, the tautness condition for Riemannian flows(i.e. $1$-dimensional Riemannian foliation) on compact manifolds are well-known, by the following result of Carri\`ere.

\begin{lem}\cite{Car, Ton}
    Let $\F$ be a Riemannian flow on a compact manifold $M$. Then $\F$ is taut if and only if $\F$ is generated by a nonvanishing Killing field.
\end{lem}

On the other hand, the following observation is straightforward. 
\begin{thm}\cite{Bes,Ton}
    Let $Y,Z$ be basic vector fields perpendicular to the given foliation $\F$ on a Riemannian manifold $(M,\F).$ Then we have the following.
    \begin{equation}
        A_YZ=\dfrac{1}{2}\pi^{\perp}[Y,Z].
    \end{equation}
\end{thm}

Thus, a Riemannian foliation $\F$ on $M$ has an involutive normal bundle $Q$ if and only if $A$ vanishes. Also, we should recall the Gauss equation and transverse Gauss equation for later use.
\begin{equation} R(\epsilon_{\alpha},\epsilon_{\beta},\epsilon_{\gamma},\epsilon_{\delta})=R^L(\epsilon_{\alpha},\epsilon_{\beta},\epsilon_{\gamma},\epsilon_{\delta})-g(T_{\epsilon_{\alpha}}\epsilon_{\gamma},T_{\epsilon_{\beta}}\epsilon_{\delta})+g(T_{\epsilon_{\alpha}}\epsilon_{\delta},T_{\epsilon_{\beta}}\epsilon_{\gamma}),
\end{equation}
\begin{equation} \begin{split}
    R(Y,Z,V,W)&=R^Q(Y,Z,V,W)-2g(A_YZ,A_VW)\\&+g(A_ZV,A_YW)-g(A_YV,A_ZW),
\end{split}
\end{equation} where $R^L$ is the Riemann curvature with respect to the leaves of a Riemannian foliation $\F$ on $M$. In particular, $R$ on $T\F$ is identified with $R^L$ on a totally geodesic Riemannian foliation, i.e. a foliation with $T=0$. 

Also, on a Riemannian foliation, the following condition should be recalled.
\begin{thm}\cite{Ton}
    Let $M$ be a complete Riemannian manifold with a Riemannian foliation $\F$. If for each vector field $Y$ orthogonal to $\F$ satisfies 
    \begin{equation}
     Ric^L(Y,Y)=\sum_{\alpha=1}^pR(Y,\epsilon_{\alpha},Y,\epsilon_{\alpha}) \geq g(A_Y\epsilon_{\alpha},A_Y\epsilon_{\alpha})
    \end{equation} then $T=0$.
\end{thm}
It should be remarked that (2.28) could be replaced with the following, by (2.27).
\begin{equation}
     Ric(Y,Y)-Ric^Q(Y,Y)+2\sum_{\alpha=1}^pg(A_Y\epsilon_{\alpha},A_Y\epsilon_{\alpha}) \geq 0.
\end{equation} 
Generally, if both $T$ and $A$ vanish, the following is proven by R. Blumenthal and J. Hebda \cite{BH}.
\begin{thm}
    Let $M$ be a complete Riemannian manifold with a foliation $\{L_p\}$ satisfying both $T=0$ and $A=0$. Then the universal covering $\tilde{M}$ of $M$ topologically splits into $\tilde{L}\times N,$ where $\tilde{L}$ is the universal cover of all leaves $L_p$ and $H$ is a manifold tangent to the transverse bundle of $\{L_p\}$.
\end{thm}
Since it is well-known that the transverse bundle of a Riemannian foliation is a totally geodesic subbundle \cite{Ton}, a simply connected, complete manifold $(M,\F)$ with a Riemannian foliation satisfying both $T=0$ and $A=0$ splits into a Riemannian product $\tilde{L}\times N$.

\begin{rmk}\normalfont
Hereafter, we focus on a compact Einstein-type manifold with Riemannian potential flow $\mathcal{F}_X$ and it is denoted by $(M,\mathcal{F}_X,g)$. Also note that if $\F_X$ satisfies $A=0,$ then it is regarded as an analogy of the gradient potential flow $\F_{\nD f}$ on a complete Einstein-type manifold, by Remark 2.3. 
\end{rmk}

\section{Curvature equations on Einstein-type manifolds}

From now on, we always assume that $\kappa$ is the basic mean curvature if there are no other explanations. Also, we examine the relation between Ricci curvature and transverse Ricci curvature on $(M,\F_X,g)$ for later use. First, let us recall the Riccati equation on an arbitrary Riemannian manifold with a Riemannian foliation $(M,\F,g)$ as follows(cf. \cite{Bes}).
\begin{equation}
    Ric(Y,Z)-Ric^Q(Y,Z)+g(T_UY,T_UZ)+2g(A_YU,A_{Z}U)-\dfrac{1}{2}L_{\kappa^{\sharp}}g(Y,Z)=0,
\end{equation} where $Y,Z$ are basic vector fields perpendicular to $X,$ and $U$ is the unit tangent vector field of $X$.

Here, if we assume that $\kappa$ is a mean curvature of $\F_X$, then the above formula is deduced as follows.
\begin{equation}
    Ric(Y,Z)-Ric^Q(Y,Z)+g(T_UY,T_UZ)+2g(A_YU,A_{Z}U)-\nD_{tr}\kappa(Y,Z)=0,
\end{equation} 

Since $\F_X$ is $1$-dimensional, 
\begin{equation}
    g(T_UY,T_UZ)=\kappa(Y)\kappa(Z)
\end{equation} holds by definition of $T$-tensor. Hence, (3.2) could be replaced with the following.

\begin{equation}
    Ric(Y,Z)-Ric^Q(Y,Z)+2g(A_YU,A_ZU)-T_{\kappa}(Y,Z)=0,
\end{equation} 

As we are assuming that the given manifold $(M,\F_X,g)$ has a Riemannian potential flow $\F_X$, we have 
\begin{equation}
    \alpha Ric(Y,Z)=-\dfrac{\beta}{2}L_Xg(Y,Z)+\lambda \text{ }g(Y,Z)=\lambda\text{ }g(Y,Z).
\end{equation} Thus, we have the following by combining (3.4), (3.5) and the definition of Riemannian foliation.
\begin{equation}
    \lambda g_Q=\alpha(Ric^Q+T_{\kappa}-2\mathcal{A}),
\end{equation} where $\mathcal{A}$ is a symmetric $2$-tensor, defined from the extension of $g(A_YU,A_ZU)$, as follows.
\begin{equation}
\begin{split}
   &\mathcal{A}(Y,Z)=g(A_YU,A_ZU)\\
   &\mathcal{A}(X,\cdot)=0, 
\end{split}
\end{equation} where $Y,Z$ are vector fields perpendicular to $\F_X$. Since $U$ is the unit vector field parallel to $X$, it is clear that $A_XU=0$ and therefore $\mathcal{A}(X,\cdot)=0$ is canonically obtained. 

Therefore, we have the following formula by taking transverse trace on (3.6).
\begin{equation}
    (n-1)\lambda=\alpha(S^Q+\text{div}_B\kappa-2|A|^2).
\end{equation}
From the above equation, the integrand of the directional derivative of (3.8) is deduced as follows.
\begin{equation}
    (n-1)\int_MX(\lambda)=\alpha\int_M X(S^Q+\text{div}_B\kappa-2|A|^2).
\end{equation}

Since $S^Q$ is basic by its definition and also $\text{div}_B\kappa$ is basic by the assumption on $\kappa$ and (2.17), the following is obtained. 
\begin{equation}
    (n-1)\int_Mg(X,\nD\lambda)=-2\alpha\int_M g(X,\nD |A|^2).
\end{equation}

Moreover, by taking the divergence of $M$ on (3.8), we have the following. 
\begin{equation}
    d\lambda-\lambda\kappa=\alpha \text{div}_B(Ric^Q+T_{\kappa})-2\alpha \text{div}\mathcal{A},
\end{equation} as $Ric^Q$ and $T_{\kappa}$ are basic.

Then the following lemma arises, which is similar to the calculation in \cite{Moo}.
\begin{lem} Let $(M,\F,g)$ be a compact Riemannian manifold with a Riemannian foliation with basic mean curvature. Then the following equation is computed.
\begin{equation}
    \text{div}_B(Ric^Q+T_{\kappa})=\dfrac{1}{2}dS^Q+d(\text{div}_B\kappa)-(\text{div}_B\kappa)\kappa.
\end{equation}
\end{lem}
\begin{proof}
By the transverse Bianchi identity and the divergence formula (2.17), we have 
\begin{equation}
    \text{div}_BRic^Q=\dfrac{1}{2}dS^Q-i_{\kappa^{\sharp}}Ric^Q.
\end{equation}
Applying the above formula in (2.19), we directly have 
\begin{equation}
    \text{div}_B(Ric^Q+T_{\kappa})=\dfrac{1}{2}dS^Q+d(\text{div}_B\kappa)-(\text{div}_B\kappa)\kappa,
\end{equation} as desired.
\end{proof}

Hence, the following is induced, by combining (3.11) and Lemma 3.1.
\begin{equation}
 d\lambda-\lambda\kappa=\dfrac{\alpha}{2}dS^Q+\alpha d(\text{div}_B\kappa)-\alpha(\text{div}_B\kappa)\kappa-2\alpha \text{div}\mathcal{A}.
\end{equation}

Therefore, the integrand of the directional derivative of (3.15) is given as follows.
\begin{equation}
   \int_M X(\lambda)=-2\alpha \int_Mg(\text{div}\mathcal{A},X^{\flat})=\alpha\int_Mg(\mathcal{A},L_Xg),
\end{equation} where $X^{\flat}$ is the $1$-form corresponding to $X$.

Since $i_X\mathcal{A}$ vanishes, we have  
\begin{equation}
   \int_M X(\lambda)=\alpha\int_Mg(\mathcal{A},L_Xg)=\alpha\int_Mg(\mathcal{A},L_Xg_Q)=0
\end{equation} and hence 
\begin{equation}
    \int_M g(X,\nD |A|^2)=0
\end{equation} is also computed from (3.10). 

Additionally, we may calculate the curvature information on leaves of the Riemannian potential flow $\F_X$ as the leaves are $1$-dimensional. To be specific, the following lemma is implied.

\begin{lem} On a compact manifold $(M,g)$ with $1$-dimensional Riemannian flow with unit tangent vector $U$, the following formula is obtained. 
   \begin{equation}
    Ric(U,U)=\text{tr}_Q(\mathcal{A})+\text{div}_B\kappa=|A|^2+\text{div}_B\kappa.
    \end{equation}
\end{lem}
\begin{proof}
    Recalling Gray-O'Neill formula (e.g. \cite{Bes,Ton}), we apply the following formula.
    \begin{equation}
       R(Y,U,Z,U)= g((D_YT)_UU,Z)-g(T_UY,T_UZ)+g((D_UA)_YZ,U)+\mathcal{A}(Y,Z).
    \end{equation}
Here, by definition of $T$ tensor, we directly calculate the following holds. 
\begin{equation}
\begin{split}
    g((D_YT)_UU,Z)&=g(D_Y\kappa^{\sharp},Z)-g(T_{D_YU}U,Z)-g(T_U(A_YU),Z)\\&
    =g(D_Y\kappa^{\sharp},Z).
\end{split}
\end{equation}
    That is, we straightforwardly have
    \begin{equation}
        \sum_{i=1}^{n-1}R(e_i,U,e_i,U)= \text{tr}_QT_{\kappa}+g((D_UA)_{e_i}e_i,U)+|A|^2,
    \end{equation} by definition of $T_{\kappa}$ and $\mathcal{A}$.
    As 
    \begin{equation}
        g((D_UA)_YZ,U)=-g((D_UA)_ZY,U)
    \end{equation} holds by definition of $A$, we obtain 
    \begin{equation} 
        Ric(U,U)=|A|^2+\text{div}_B\kappa,
    \end{equation} as desired.
\end{proof}

\begin{cor}Any compact manifold $(M,g)$ with $1$-dimensional Riemannian flow satisfies the following.
\begin{equation}
    S^Q-S=|A|^2-2\text{div}_B\kappa.
\end{equation}
\end{cor}
\begin{proof}
    The above formula is derived straightforwardly, by putting (3.24) in the transverse trace of (3.4).
\end{proof}

On the other hand, along the leaves of $\F_X,$ we have the following formula on a general $(M,\F_X,g)$.
\begin{equation}
    \lambda=\alpha Ric(U,U)+\dfrac{\beta}{2}L_Xg(U,U)+{\gamma}|X|^2.
\end{equation}
As $\F_X$ is Riemannian, the following is deduced from (3.26), since we apply (3.24) and the definition of Riemannian foliation.
\begin{equation}
    {\lambda}=\alpha |A|^2+\alpha\text{div}_B\kappa+{\beta}\text{div}X+{\gamma}|X|^2.
\end{equation}

Therefore, the following integrand formula is derived, by (3.17), (3.18), and the assumption that $\kappa$ is basic.
\begin{equation}
    {\beta}\int_MX(\text{div}X)+{\gamma}\int_MX(|X|^2)=0.
\end{equation}
In other words, \begin{equation}
   {\gamma}\int_MX(|X|^2)= {\beta}\int_M|\text{div}X|^2.
\end{equation}

We also obtain the following by taking divergence on (2.10). 
\begin{equation}
    \dfrac{\alpha}{2}dS+\dfrac{\beta}{2}\text{div}(L_Xg)+\gamma(D_XX+(\text{div}X)X)=d\lambda.
\end{equation}

That is, 
\begin{equation}
    \dfrac{\alpha}{2}d(S^Q-|A|^2+2\text{div}_B\kappa)+\dfrac{\beta}{2}\text{div}(L_Xg)+\gamma(D_XX+(\text{div}X)X)=d\lambda
\end{equation} is obtained by (3.25).

Therefore, we obtain 
\begin{equation}
    0=\int_Mg(\nD\lambda,X)=\dfrac{\beta}{2}\int_M g(\text{div}(L_Xg),X^{\flat})+\gamma\int_Mg(D_XX,X)+\gamma\int_M\text{div}X|X|^2.
\end{equation} from the above, as $S^Q$ and $\kappa$ are basic and (3.17), (3.18) are applied.

As $-\text{div}$ is the formal adjoint of the Riemannian connection, we reach the following formula. 
\begin{equation}
   \dfrac{\beta}{4} \int_M|L_Xg|^2=-\dfrac{\gamma}{2}\int_MX(|X|^2).
\end{equation}

\section{Properties on Einstein-type manifolds}

In this section, we prove the main results of this paper. First, the following theorem is identified with Theorem 1.1.
\begin{thm}
Let $(M,g)$ be a compact Einstein-type manifold of dimension $\geq 3$ with the Riemannian potential flow $\F_X$. If $\beta\neq 0$ and the mean curvature $\kappa$ of $F_X$ is basic, then $\F_X$ is taut.
\end{thm}
\begin{proof}
    By (3.29) and (3.33), the following is derived.
    \begin{equation}
        \dfrac{1}{2} \int_M|L_Xg|^2=-\int_M|\text{div}X|^2
    \end{equation} Thus, $L_Xg$ is forced to vanish by positive definiteness. Therefore, $\F_X$ should be taut(cf. Lemma 2.8). 
\end{proof}

Now let us assume that a compact Einstein-type manifold with Riemannian potential flow $(M,\F_X,g)$ satisfies $A=0$ and has basic $\kappa$. i.e. the transverse bundle $Q$ is involutive(cf. Lemma 2.9). Although the given foliation $\F_X$ is taut, it is not determined whether $\kappa$ vanishes. For this reason, we consider the following equation. 
\begin{equation}
    0=2L_Xg(U,Y)=g(D_UX,Y)+g(D_YX,U),
\end{equation} where $Y$ is an arbitrary vector field perpendicular to the nonvanishing potential vector field $X$. Then 
\begin{equation}
g(D_UX,Y)+g(D_YX,U)=g(U(|X|)U,Y)+|X|g(\kappa^{\sharp},Y)+g(A_YX,U)=|X|g(\kappa^{\sharp},Y)\end{equation} is clearly obtained. Therefore, we have $\kappa=0$ in this situation, as $|X|$ is nonvanishing. That is, both $T$ and $A$ vanish on such compact Einstein-type manifolds since $T$ and $\kappa$ are identified on a Riemannian flow. 

Hence, we are ready to prove Corollary 1.2.
\begin{cor}
    Let $(M,g)$ be a compact Einstein-type manifold of dimension $\geq 3$ with the Riemannian potential flow $\F_X$. If we assume $\alpha\beta \neq0,$ $A=0$ and $\kappa$ is basic, then the Riemannian universal covering $\tilde{M}$ of $M$ splits into $\mathbb{R} \times N,$ for some complete Einstein manifold $N$.
\end{cor}
\begin{proof}
    By Theorem 2.11, $T=0$ and $A=0$ directly implies the following Riemannian product. 
    \begin{equation}
        \tilde{M}=\mathbb{R} \times N
    \end{equation} for some complete manifold $N$. Therefore, we only need to prove that $N$ is an Einstein manifold. 
    Since we have \begin{equation}
        \lambda g_Q=\alpha Ric^Q 
    \end{equation} by (3.6), the transverse metric is clearly transverse Einstein, as desired.
\end{proof}
\begin{rmk}\normalfont
If we replace $\alpha\beta\neq 0$ with $\alpha=0$ and $\beta \neq 0$ in Corollary 4.2, then (2.10) implies the following. 
\begin{equation}
    \gamma \text{ }X\otimes X=0.
\end{equation} As $X$ is assumed to be nonvanishing, $\gamma=0$ is forced. Therefore, we could summarize that $\tilde{M}$ splits into $\mathbb{R}\times N$ for some complete Riemannian manifold $N$ as $T=0$ and $A=0$, however $N$ is not determined to be Einstein as (4.6) has no information on transverse geometry.   
\end{rmk}

\section{Applications on almost Ricci solitons}
In this section, we prove Corollary 1.3 and Theorem 1.4. In fact, the following remark proves Corollary 1.3.
\begin{rmk}\normalfont(Proof of Corollary 1.3)
Note that an almost Ricci soliton is a Einstein-type manifold satisfying $\beta\neq0$, since the definition of such metric is given as follows.
\begin{equation}
    Ric+\dfrac{1}{2}L_Xg=\lambda \text{ }g,
\end{equation}for some smooth function $\lambda$. Thus, we could directly apply Theorem 1.1 and $L_Xg=0$ is obtained. Therefore, we have \begin{equation}
     Ric=\lambda g
\end{equation} with $\lambda$ is forced to be constant since we are considering $n\geq 3$-dimensional Riemannian manifolds. In particular, if $A=0$ is additionally assumed, then it is available to combine the above observation and Corollary 1.2. This proves Corollary 1.3.
\end{rmk}

On the other hand, we recall the justification of the result of Ranjan and the author as follows.
\begin{lem}\cite{Ran}
    A compact Riemannian manifold $M$ with strictly negative Ricci curvature cannot attain a Riemannian flow.
\end{lem}
\begin{proof}
    For the sake of completeness of this paper, let us remind of the proof. Note that (3.22) is obtained as the following form, since $\kappa$ may not be divergence-free.
    \begin{equation}
       Ric(U,U)=\text{div}_Q\kappa-|\kappa|^2+|A|^2=\text{div}\kappa+|A|^2,
    \end{equation} as the following is calculated directly. 
    \begin{equation}
        \text{div}_Q\kappa=\text{div}\kappa-g(D_U\kappa^{\sharp},U)=\text{div}\kappa+|\kappa|^2,
    \end{equation} for the unit vector $U$ tangent to the Riemannian flow.
    Therefore, the integrand of the above equation is given as follows. 
    \begin{equation}
        \int_M Ric(U,U)= \int_M |A|^2 <0,
    \end{equation} which occurs a contradiction. Therefore, the assertion is justified.
\end{proof}
\begin{lem}\cite{Moo}
    Let $M$ be a compact Riemannian manifold endowed with a transverse Einstein foliation $\F$. If the transverse scalar curvature $S^Q$ of $\F$ is nonnegative, then $\F$ is taut.
\end{lem}
\begin{proof}
    As the tautness of a Riemannian foliation is a topological property, we do not need to consider the metric of ambient space $(M,g)$. Thus, we may assume that $\text{div}_B\kappa$ vanishes. Then by (2.24), we have the following.
    \begin{equation}
        \text{div}_BT_{\kappa}=i_{\kappa^{\sharp}}Ric^Q.
    \end{equation} That is, the following is straightforward.
    \begin{equation}
        \int_MRic^Q(\kappa^{\sharp},\kappa^{\sharp})=-\int_M g(T_{\kappa},\nD_{tr}\kappa)=-\int_M|\nD_{tr}\kappa|^2.
    \end{equation}
    As a consequence, $\nD_{tr}\kappa=0$ is forced. That is, $\kappa=0$ is induced, since we have $\text{div}_Q\kappa=0$.
\end{proof}

\begin{thm}
    A transverse Einstein foliation of $1$-dimensional leaves on a compact Einstein manifold is taut.
\end{thm}
\begin{proof}
    In virtue of completeness of this paper, we sketch the proof.
    By Lemma 5.3, a taut Riemannian foliation must have a nonnegative $Ric^Q(\kappa^{\sharp},\kappa^{\sharp})$. Therefore, we may assume that the given transverse Einstein foliation have a strictly negative $Ric^Q=\lambda^Q \text{ } g_Q$. Also, assume $Ric=\lambda' \text{ }g$ for some constant $\lambda'$. 
    Since the transverse trace of (3.1) and (3.3) imply that 
    \begin{equation}
       (n-1)(\lambda'-\lambda^Q)-\text{div}\kappa+2|A|^2=0,
    \end{equation} we have
    \begin{equation}
        -2\int_M |A|^2=(n-1)(\lambda'-\lambda^Q)\text{vol}(M).
    \end{equation} Therefore, $\lambda^Q \geq \lambda'$ is obtained(cf. \cite{Pak} for taut foliation case only). Since the non-existence of a Riemannian flow on a compact manifold with strictly negative Ricci curvature is guaranteed by Lemma 5.2, we may assume that $\lambda'$ is nonnegative. That is, $\lambda^Q$ is automatically determined nonnegative. Hence, the assertion is proven.
\end{proof}

It is canonical to wonder that how to extend Theorem 5.4 to an arbitrary dimensional Riemannian foliation. As a consequence, we have a partial result of the above question as follows.
\begin{thm}
    A transverse Einstein foliation with scalar-flat(i.e. the scalar curvature of all leaves vanish) leaves on a compact Einstein manifold must be taut.
\end{thm}

Before we prove, we need to check the following lemma.
\begin{lem}
    Let $(M,\F,g)$ be a compact Einstein manifold with a transverse Einstein foliation. If we assume that all leaves of $\F$ are scalar-flat, then the following equation holds.
    \begin{equation}
    \int_M|A|^2=p\int \lambda',
    \end{equation} where $\lambda'$ is the constant satisfying $Ric=\lambda'\text{ }g$.
\end{lem}
\begin{proof}
    Since the given Riemannian manifold is Einstein, the Ricci decomposition (2.3) of $M$ is derived as follows.
    \begin{equation}
        R=\dfrac{\lambda'}{2(n-1)}g\KN g+W.
    \end{equation}
    Also, we may consider the following tensor inner product from (2.27).
    \begin{equation}
        g(\dfrac{\lambda'}{2(n-1)}g\KN g+W,g_Q\KN g_Q)=4S^Q-12|A|^2,
    \end{equation} since the following calculations are computed.
    \begin{equation}
        -4g(A_{e_i}e_j,A_{e_k}e_l)(g(e_i,e_k)g(e_j,e_l)-g(e_i,e_l)g(e_j,e_k))=-8|A|^2,
    \end{equation}
    \begin{equation}
        2g(A_{e_i}e_l,A_{e_j}e_k)(g(e_i,e_k)g(e_j,e_l)-g(e_i,e_l)g(e_j,e_k))=-2|A|^2,
    \end{equation}
    \begin{equation}
        -2g(A_{e_i}e_k,A_{e_j}e_l)(g(e_i,e_k)g(e_j,e_l)-g(e_i,e_l)g(e_j,e_k))=-2|A|^2,
    \end{equation} and by definition of Kulkarni-Nomizu product and the transverse scalar curvature, 
    \begin{equation}
   g(R^Q,g_Q\KN g_Q )=4S^Q 
    \end{equation} holds. Also, 
    \begin{equation}
        g(g\KN g,g_Q\KN g_Q)=8q(q-1)
    \end{equation} is derived by direct calculation of metric contraction and Kulkarni-Nomizu product.
    Thus, we have \begin{equation}
        \dfrac{4q(q-1)}{(n-1)}\lambda'+g(W,g_Q\KN g_Q)=4S^Q-12|A|^2.
    \end{equation}
    In other words, 
     \begin{equation}
        \dfrac{1}{4}g(W,g_{\F}\KN g_{\F})=S^Q-3|A|^2-\dfrac{q(q-1)}{(n-1)}\lambda'
    \end{equation} as $g$ is bundle-like and $W$ is trace-free.
Hence, \begin{equation}
\begin{split}
     \dfrac{1}{4}g(W,g_{\F}\KN g_{\F})&=q(\lambda^Q-\lambda')-3|A|^2+\dfrac{pq}{(n-1)}\lambda'\\&
     =|T|^2+2|A|^2-\text{div}\kappa-|\kappa|^2-3|A|^2+\dfrac{pq}{(n-1)}\lambda',
\end{split} 
\end{equation} by the assumption and (3.1).
    
    On the other hand, the following is deduced from (2.26) and the assumption.
    \begin{equation}
        g(R,g_{\F}\KN g_{\F})=g(R^L,g_{\F}\KN g_{\F})-4|\kappa|^2+4|T|^2=-4|\kappa|^2+4|T|^2,
    \end{equation}
    since we also have the following.
    \begin{equation}
-g(T_{\epsilon_{\alpha}}\epsilon_{\gamma},T_{\epsilon_{\beta}}\epsilon_{\delta})(g(\epsilon_{\alpha},\epsilon_{\gamma})g(\epsilon_{\beta},\epsilon_{\delta})-g(\epsilon_{\alpha},\epsilon_{\delta})g(\epsilon_{\beta},\epsilon_{\gamma}))=-|\kappa|^2+|T|^2,
\end{equation} 
\begin{equation}
g(T_{\epsilon_{\alpha}}\epsilon_{\delta},T_{\epsilon_{\beta}}\epsilon_{\gamma})(g(\epsilon_{\alpha},\epsilon_{\gamma})g(\epsilon_{\beta},\epsilon_{\delta})-g(\epsilon_{\alpha},\epsilon_{\delta})g(\epsilon_{\beta},\epsilon_{\gamma}))=|T|^2-|\kappa|^2.
    \end{equation}

That is,
\begin{equation}
     \dfrac{1}{4}g(W,g_{\F}\KN g_{\F})=-\text{div}\kappa-|A|^2+\dfrac{pq}{(n-1)}\lambda'+\dfrac{1}{4}g(R,g_{\F}\KN g_{\F}) 
\end{equation} is obtained.
From the above equation, we derive
\begin{equation}
    \text{div}\kappa+|A|^2-\dfrac{pq}{(n-1)}\lambda'=\dfrac{1}{4}g(\dfrac{\lambda'}{2(n-1)}g\KN g,g_{\F}\KN g_{\F}).
\end{equation}
i.e. \begin{equation}
    \text{div}\kappa+|A|^2-\dfrac{pq}{(n-1)}\lambda'=\dfrac{8p(p-1)}{8(n-1)}\lambda'
\end{equation} is calculated, by definition of $g_{\F}$.
Therefore, the integrand of the above formula is given as follows.
\begin{equation}
    \int_M|A|^2=\dfrac{p(n-1)}{(n-1)}\int_M\lambda'=p\int_M\lambda'=p\lambda'\text{vol}(M),
\end{equation} as desired.
\end{proof}

Now we are ready to prove Theorem 5.5.
\begin{proof} Let $\lambda'$ be a constant satisfying $Ric=\lambda\text{ }g$ on a compact Einstein manifold $(M,\F,g)$ with a transverse Einstein foliation $\F$ satisfying $Ric^Q=\lambda^Q\text{ }g_Q$. By (5.11), it is clear that the constant $\lambda'$ is nonnegative. On the other hand, if $\lambda' \geq \lambda^Q$, then we have
\begin{equation}
    (\lambda'-\lambda^Q)g(Y,Y)+2\sum_{\alpha=1}^pg(A_Y\epsilon_{\alpha},A_Y\epsilon_{\alpha}) \geq 0.
\end{equation} By Theorem 2.10, $\F$ should be totally geodesic and the equality holds in (5.29), since we generally have
\begin{equation}
  Ric(Y,Z)-Ric^Q(Y,Z)+\sum_{\alpha=1}^p(g(T_{\epsilon_{\alpha}}Y,T_{\epsilon_{\alpha}}Z)+2g(A_Y{\epsilon_{\alpha}},A_{Z}{\epsilon_{\alpha}}))-\dfrac{1}{2}L_{\kappa^{\sharp}}g(Y,Z)=0,   
\end{equation} for an arbitrary Riemannian foliation on complete manifolds \cite{Bes}.

Hence, $\lambda' \leq \lambda^Q$ is forced since $\lambda'\geq \lambda^Q$ implies $A=0$ and $\lambda'=\lambda^Q$. That is, $\lambda^Q$ is a nonnegative constant. Therefore, $\F$ must be taut by Lemma 5.3.
\end{proof}

\addresses


\begin{thebibliography}{00}
\bibitem{AL} J. A. Alvarez L\'opez, \textit{The basic component of the mean curvature of Riemannian foliations,} Ann. Glob. Anal. Geom. 10, 179–194 (1992).
\bibitem{Bes} A. Besse, \textit{Einstein manifolds,} Springer-Verlag, New York, (1987).
\bibitem{BD} P. Baird and L. Danielo \textit{Three-dimensional Ricci solitons which project to surfaces,} J. Reine Angew. Math. 608, 65–91 (2007). 
\bibitem{BH} R. A. Blumenthal and J. J. Hebda, \textit{De Rham decomposition theorems for
foliated manifolds,} Ann. Inst. Fourier, Grenoble, 33, 2, 183-198 (1983).
\bibitem{Car} Y. Carri\`ere, \textit{Flots riemanniens,} Transversal structure of foliations(Toulouse, 1982), 31–52, (1984).
\bibitem{CM} G. Catino and P. Mastrolia, \textit{Weyl scalars on compact Ricci solitons, }J. Geom. Anal. 29, no. 4, 3328–3344 (2019).
\bibitem{CMM} G. Catino, P. Mastrolia and D. D. Monticelli, \textit{Gradient Ricci solitons with vanishing conditions on Weyl,} J. Math. Pures Appl. (9) 108, no. 1, 1–13 (2017).
\bibitem{CMMR} G. Catino, P. Mastrolia, D. D. Monticelli, and M. Rigoli, \textit{On the geometry of gradient Einstein-type manifolds,} Pacific J. Math. 286, no. 1, 39–67 (2017).
\bibitem{CH} J. Co and S. Hwang, \textit{Gradient almost Ricci solitons with vanishing conditions on Weyl curvature and Bach tensor,} J. Korean Math. Soc. 57, no. 2, 539–552 (2020).
\bibitem{FG} M. Fern\'andez-L\'opez and E. Garc\'ia-R\'io, \textit{Rigidity of shrinking Ricci solitons,} Math. Z. 269, no. 1-2, 461–466 (2011).
\bibitem{Jun} S.D. Jung, \textit{Eigenvalue estimates for the basic Dirac operator on a Riemannian foliation admitting a basic harmonic 1-form,} J. Geom. Phys. 57(4), 1239–1246,
(2007). 
\bibitem{Lau} J. Lauret, \textit{Ricci soliton solvmanifolds,} J. Reine Angew. Math. 650, 1–21 (2011).
\bibitem{Moo} J. Moon, \textit{Symmetric tautness tensor of Riemannian foliations,} arXiv: 2605.25408 (2026).
\bibitem{Ton} Ph. Tondeur, \textit{Geometry of foliations,} Monographs in Mathematics, vol. 90, p. 305. Birkhäuser Verlag, Basel (1997).
\bibitem{Pak} H. K. Pak, \textit{On one-dimensional metric foliations in Einstein spaces,} Illinois J. Math. 36, no. 4, 594–599 (1992).
\bibitem{Ran} A. Ranjan, \textit{Structural equations and an integral formula for foliated manifolds,} Geom. Dedicata 20, no. 1, 85–91 (1986).
\bibitem{Rei} B. L. Reinhart, \textit{Foliated manifolds with bundle-like metrics,} Ann. Math. 2(69), 119–132 (1959). 





\end{thebibliography}
\end{document}